\documentclass{amsart}          

\usepackage{amsmath,amsthm}     
\usepackage{amssymb}            
\usepackage{euscript}           
\usepackage{enumerate,calc}
\usepackage[matrix,arrow,curve,frame]{xy}    

\usepackage{hyperref}

\xymatrixcolsep{1.9pc}                          
\xymatrixrowsep{1.9pc}
\newdir{ >}{{}*!/-5pt/\dir{>}}                  

\numberwithin{equation}{section}

\newtheorem{thm}[subsection]{Theorem}

\newtheorem{prop}[subsection]{Proposition}

\newtheorem{cor}[subsection]{Corollary}
\newtheorem{lemma}[subsection]{Lemma}
\newtheorem*{thm*}{Theorem}
\newtheorem*{conj*}{Conjecture}
\newtheorem*{lemma*}{Lemma}

\theoremstyle{definition}  
\newtheorem{example}[subsection]{Example}

\newtheorem{remark}[subsection]{Remark}

  {\end{list}}

\newcommand{\chapter}{\section}

\newcommand{\Smash}             {\wedge}

\newcommand{\cat}{\EuScript}    

\newcommand{\cU}{{\cat U}}

\newcommand{\field}[1]  {\mathbb #1} 
\newcommand{\A}         {\field A}

\newcommand{\R}         {\field R}

\newcommand{\Z}         {\field Z}

\newcommand{\bP}        {\field P}

\DeclareMathOperator{\Ext}{Ext}

\newcommand{\ra}{\rightarrow}                   

\newcommand{\period}    {{\makebox[0pt][l]{\hspace{2pt} .}}}
\newcommand{\comma}     {{\makebox[0pt][l]{\hspace{2pt} ,}}}

\newcommand{\tuborg}{\left\{\begin{array}{ll}}
\newcommand{\sluttuborg}{\end{array}\right.}

\begin{document}

\title[On the equivariant cohomology of rotation groups]{On the equivariant cohomology of rotation groups and Stiefel manifolds}
\author{William Kronholm}
\address{Department of Mathematics\\ Whittier College\\ Whittier, CA 90608 }

\begin{abstract}In this paper, we compute the $RO(\mathbb{Z}/2)$-graded equivariant cohomology of rotation groups and Stiefel manifolds with particular involutions.
\end{abstract}

\date{\today}

\maketitle

\tableofcontents

\vspace*{-0.4in}
\section*{Introduction} Lewis, May, and McClure extend ordinary Bredon cohomology theories to $RO(G)$-graded theories in \cite{LMM}. Equivariant cohomology and homotopy theories have had a wide array of applications, in part because of their close connections with Voevodsky's motivic cohomology theory. In spite of this, there have been very few published computations of the equivariant cohomology of specific spaces. Some examples of computations are in \cite{FREENESS}, \cite{ROGSS}, and \cite{Shulman}. This paper provides computations of an important class of $\Z/2$-spaces. In particular, we compute the cohomology of the special orthogonal groups and Stiefel manifolds with particular $\Z/2$ actions. The main results are stated in Theorem \ref{thm:SOpq_algebra} and Theorem \ref{thm:stiefel_algebra}.

Section \ref{sec:prelim} provides some background information on $RO(\Z/2)$-graded cohomology and establishes some definitions and notation which will be used throughout the paper.

In Section \ref{sec:rotation} we introduce an equivariant cell structure on $SO(p,q)$, the group of rotations of $p$-dimensional Euclidean space endowed with a particular action of $\Z/2$, and use this cell structure to determine the cohomology of $SO(p, q)$ as an algebra over the cohomology of a point with constant $\underline{\Z/2}$ Mackey functor coefficients.

In Section \ref{sec:stiefel} we use the cell structure on $SO(p, q)$ from section \ref{sec:rotation} to put an equivariant cell structure on $V_q(\R^{p, q})$, the Stiefel manifold of $q$-frames in the $\Z/2$-representation $\R^{p, q}$ with action inherited from $\R^{p, q}$. The cell structure on the Stiefel manifold is compatible with the one on $SO(p, q)$ and allows for the cohomology algebra structure of $V_q(\R^{p, q})$ to be deduced from that of $SO(p, q)$. 

Finally, we discuss some possible further directions of study in Section \ref{sec:further}.

\section{Preliminaries}
\label{sec:prelim} This section contains some of the background information
and notations that will be used throughout the paper.  In this
section, $G$ can be any finite group unless otherwise specified.

Given a $G$-representation $V$, let $D(V)$ and $S(V)$ denote the
unit disk and unit sphere, respectively, in $V$ with action induced
by that on $V$.  A \bf$\mathrm{Rep}(G)$-complex \rm is a $G$-space $X$ with a
filtration $X^{(n)}$ where $X^{(0)}$ is a disjoint union of
$G$-orbits and $X^{(n)}$ is obtained from $X^{(n-1)}$ by attaching
cells of the form $D(V_\alpha)$ along maps $f_\alpha \colon
S(V_\alpha) \ra X^{(n-1)}$ where $V_\alpha$ is an $n$-dimensional
real representation of $G$.  The space $X^{(n)}$ is referred to as
the \bf $n$-skeleton \rm of $X$, and the filtration is referred to as a \bf cell
structure\rm . In addition, $S^V$ will denote the one-point compactification of $V$ with fixed base point at infinity. If $X$ is a $\mathrm{Rep}(G)$-complex, then $X^{(n)}/X^{(n-1)} \cong \bigvee_i S^{V_i}$ where each $V_i$ is an $n$-dimensional $G$-representation.

The reader is referred to \cite{Alaska} for details on $RO(G)$-graded cohomology theories. Briefly, these are theories graded on the Grothendieck ring of virtual representations of the group $G$. The natural coefficients for such theories are Mackey functors. In fact, $RO(G)$-graded cohomology theories can naturally be thought of as ($RO(G)$-graded) Mackey functor valued. However, this paper will focus on abelian group valued cohomology theories.

According to \cite{Alaska}, each Mackey functor $M$ uniquely
determines an $RO(G)$-graded cohomology theory characterized by
\begin{enumerate}
\item $H^n(G/H;M) =\begin{cases}
M(G/H) & \text{ if } n=0 \\
0 & \text{otherwise}\end{cases}$
\item The map $H^0(G/K;M) \ra H^0(G/H;M)$ induced by $i \colon G/H \ra G/K$ is the transfer map $i^*$ in the Mackey functor.
\end{enumerate}
Here, $n$ denotes the $n$-dimensional trivial $G$-representation.

In this paper, the group $G$ will be $\Z/2$. For the precise definition of a $\Z/2$-Mackey functor, the
reader is referred to \cite{LMM} or \cite{DuggerKR}.  The
data of a $\Z/2$-Mackey functor $M$ are encoded in a diagram like the one
below.

\[\xymatrix{ M(\Z/2) \ar@(ur,ul)[]^{t^*} \ar@/^/[r]^(0.6){i_*} & M(e) \ar@/^/[l]^(0.4){i^*}} \]

The maps must satisfy the following four conditions:
\begin{enumerate}
\item $(t^*)^2 = id$,
\item $t^*i^*=i^*$,
\item $i_*(t^*)^{-1}=i_*$,
\item $i^*i_*=id+t^*$.
\end{enumerate}

A $p$-dimensional real $\Z/2$-representation $V$ decomposes as
$V=(\R^{1,0})^{p-q} \oplus (\R^{1,1})^q =\R^{p,q}$ where $\R^{1,0}$
is the trivial 1-dimensional real representation of $\Z/2$ and $\R^{1,1}$ is the nontrivial
1-dimensional real representation of $\Z/2$.  Thus the $RO(\Z/2)$-graded theory is
a bigraded theory, one grading measuring dimension and the other
measuring the number of ``twists''.  In this case, we write
$H^{V}(X;M)=H^{p,q}(X;M)$ for the $V^{\text{th}}$ graded component
of the $RO(\Z/2)$-graded equivariant cohomology of $X$ with
coefficients in a Mackey functor $M$. Similarly, we will write $S^{p,q}$ for $S^V$ when $V=\R^{p,q}$.

In this paper, the Mackey functor will always be constant $M=\underline{\Z/2}$ which has the
following diagram.

\[\xymatrix{ \Z/2 \ar@(ur,ul)[]^{id} \ar@/^/[r]^{0} & \Z/2 \ar@/^/[l]^{id}} \]

\noindent Because this paper only considers this constant Mackey functor, the coefficients will be suppressed from the notation and we write $H^{p,q}(X)$ for $H^{p,q}(X;\underline{\Z/2})$.

With these constant coefficients, the $RO(\Z/2)$-graded cohomology
of a point is given by the picture in Figure \ref{fig:pt}.

\begin{figure}[htpb]
\centering
\begin{picture}(100,100)(-100,-100)
\put(-100,-50){\vector(1,0){100}}
\put(-50,-100){\vector(0,1){100}}

\put(-50, -51){\line(0,1){40}} \put(-50, -51){\line(1,1){40}}
\put(-50, -91){\line(0,-1){20}} \put(-50, -91){\line(-1,-1){20}}
\put(-50, -51){\circle*{2}}

\multiput(-90,-51)(20,0){5}{\line(0,1){3}}
\multiput(-52,-90)(0,20){5}{\line(1,0){3}}
\put(-49,-57){$\scriptscriptstyle{0}$}
\put(-29,-57){$\scriptscriptstyle{1}$}
\put(-9,-57){$\scriptscriptstyle{2}$}
\put(-69,-57){$\scriptscriptstyle{-1}$}
\put(-89,-57){$\scriptscriptstyle{-2}$}
\put(-57,-47){$\scriptscriptstyle{0}$}
\put(-57,-27){$\scriptscriptstyle{1}$}
\put(-57,-7){$\scriptscriptstyle{2}$}

\put(-48,-35){$\tau$} \put(-30,-35){$\rho$} \put(-10,-15){$\rho^2$}
\put(-30,-17){$\tau\rho$} \put(-58, -90){$\theta$} \put(-48,
-102){$\frac{\theta}{\tau}$} \put(-78, -102){$\frac{\theta}{\rho}$}

\put(-58,0){${q}$} \put(0,-60){${p}$}

\multiput(-50, -30)(20,20){2}{\circle*{2}} \multiput(-50,
-50)(20,20){3}{\circle*{2}} \put(-50, -10){\circle*{2}} \put(-50,
-90){\circle*{2}} \put(-50, -110){\circle*{2}} \put(-70,
-110){\circle*{2}}

\end{picture}

\caption{$H^{*,*}(pt;\Z/2)$} \label{fig:pt}
\end{figure}

Every lattice point in the picture that is inside the indicated
cones represents a group isomorphic to $\Z/2$.  The \bf top cone \rm is a
polynomial algebra on the nonzero elements $\rho \in
H^{1,1}(pt)$ and $\tau \in
H^{0,1}(pt)$.  The nonzero element $\theta \in H^{0,-2}(pt)$ in the \bf bottom
cone \rm is infinitely divisible by both $\rho$ and $\tau$.  The
cohomology of $\Z/2$ is easier to describe:
$H^{*,*}(\Z/2)=\Z/2[t, t^{-1}]$ where $t \in
H^{0,1}(\Z/2)$.  Details can be found in
\cite{DuggerKR} and \cite{Caruso}. From here on out, we will denote $H^{*,*}(pt; \underline{\Z/2})$ by $H\underline{\Z/2}$.

Note that the suspension axioms completely determine the cohomology for the spheres $S^{p,q}$. If $(p,q) \neq (1,1)$, then $H^{*,*}(S^{p,q})=H\underline{\Z/2}[x]/x^2$ where $x$ is in bidegree $(p,q)$. In the special case of $S^{1,1}$ we have the following proposition. (See also \cite{FREENESS}.)

\begin{prop}
As a $H\underline{\Z/2}$-module, $H^{*,*}(S^{1,1})$ is free with a single
generator $a$ in degree $(1,1)$.  As an algebra, $H^{*,*}(S^{1,1}) \cong
H\underline{\Z/2}[a]/(a^2 = \rho a)$.
\end{prop}

\begin{proof}

The statement about the module structure is immediate.

Since $S^{1,1}$ is a $K(\Z(1),1)$, we can consider $a \in
[S^{1,1},S^{1,1}]$ as the class of the identity and $\rho \in [pt,
S^{1,1}]$ as the inclusion. Then $a^2$ is the composite

$\xymatrix{a^2 \colon S^{1,1} \ar[r]^(0.45)\Delta & S^{1,1} \Smash
S^{1,1} \ar[r]^(0.6){a \Smash a} & S^{2,2} \ar[r] & K(\Z/2(2),2)}$.

\noindent Similarly, $\rho a$ is the composite

$\xymatrix{\rho a \colon S^{1,1} \ar[r] & S^{0,0} \Smash S^{1,1}
\ar[r]^(0.6){\rho \Smash a} & S^{2,2} \ar[r] & K(\Z/2(2),2)}$.

\noindent The claim is that these two maps are homotopic.
Considering the spheres involved as one point compactifications of
the corresponding representations, the map $a^2$ is inclusion of
$(\R^{1,1})^+$ as the diagonal in $(\R^{2,2})^+$ and $\rho a$ is
inclusion of $(\R^{1,1})^+$ as the vertical axis.  There is then an
equivariant homotopy $H\colon (\R^{1,1})^+ \times I \ra
(\R^{2,2})^+$ between these two maps given by $H(x,t) = (t x, x)$.

\end{proof}

A useful tool is the following exact sequence of \cite{AM}.

\begin{lemma}[Forgetful Long Exact Sequence]
\label{lemma:forget} Let $X$ be a based $\Z/2$-space.  Then for
every $q$ there is a long exact sequence

$$\xymatrix{\cdots \ar[r] & H^{p,q}(X) \ar[r]^(0.4){\cdot \rho} & H^{p+1,q+1}(X) \ar[r]^(0.55){\psi} & H^{p+1}(X) \ar[r]^(0.45)\delta & H^{p+1,q}(X)} \ra \cdots\period$$

\end{lemma}

\noindent The map $\cdot \rho$ is multiplication by $\rho \in
H^{1,1}(pt)$ and $\psi$ is the forgetful map to
non-equivariant cohomology with $\Z/2$ coefficients.

Given a filtered $\Z/2$ space $X$, for each fixed $q$ there is a long exact sequence

\[ \cdots H^{*,q}(X^{(n+1)}/X^{(n)})\ra H^{*,q}(X^{(n+1)}) \ra H^{*,q}(X^{(n)}) \ra H^{*+1,q}(X^{(n+1)}/X^{(n)}) \cdots \]

\noindent and so there is one spectral sequence for each integer $q$.  The
specifics are given in the following proposition.

\begin{prop} Let $X$ be a filtered $\Z/2$-space.  Then for each $q\in\Z$ there is a spectral sequence with
$$E_1^{p,n} = H^{p,q}(X^{(n+1)}, X^{(n)})$$
converging to $H^{p,q}(X)$.
\end{prop}

The construction of the spectral sequence is completely standard. (See, for example, Proposition 5.3 of \cite{McC}.) When the space $X$ is filtered in such a way that $X^{(n+1)}$ is obtained from $X^{(n)}$ by attaching cells, the above spectral sequence will be called the \bf cellular spectral sequence\rm. The cellular spectral sequence will be used extensively in the computations below.

\section{Rotation groups}\label{sec:rotation} In this section, we define a $\mathrm{Rep}(\Z/2)$-complex structure on the group of rotations $SO(n)$ with a particular action of $\Z/2$. The construction closely follows that of \cite{Hatcher}, which in turn is inspired by \cite{Miller} and \cite{WhiteheadSpheres}. The key is to introduce the correct action of $\Z/2$ on the rotation groups so that the standard constructions are equivariant.

 Let $O(p,q)$ denote the group of orthogonal transformations of $\R^p$ with the following $\Z/2$-action.\footnote{The notation is similar to that sometimes used for the indefinite orthogonal group, but the spaces are different. Here, $O(p,q)$ denotes the \bf full \rm orthogonal group with a particular group action, while the indefinite orthogonal group is a proper subgroup of the orthogonal group.} Let $g$ denote the non-identity element of $\Z/2$. Denote by $I_{p,q}$ the block matrix
\[
I_{p,q}=\begin{pmatrix}I_{p-q}&0\\0&-I_q\end{pmatrix}\comma
\]
where $I_n$ is the $n\times n$-identity matrix. Each orthogonal transformation of $\R^p$ can be represented as an orthogonal matrix $A$ with $\det A=\pm 1$. Define
\[
g\cdot A = I_{p,q}AI_{p,q}\period
\]

\noindent That is, $g$ acts on $A$ by changing the sign on the last $q$ entries of each row and each column. If we represent $A$ as a block matrix
\[
A=\begin{pmatrix}A_1 & A_2\\A_3&A_4\end{pmatrix}\comma
\]
where $A_1$ is $(p-q)\times(p-q)$ and $A_4$ is $q\times q$, then
\[
g\cdot A=\begin{pmatrix}A_1 & -A_2\\-A_3&A_4\end{pmatrix}\period
\]
Notice that $\det(g\cdot A) = \det(I_{p,q})\det(A)\det(I_{p,q}) =\det(A)$ and so $g$ preserves determinant. Also, it is easy to check that if $A,B\in O(p,q)$, then $g\cdot (AB)=(g\cdot A)(g\cdot B)$. In particular, this action of $g$ induces an action on the subgroup of rotations which will be denoted by $SO(p,q)$.

Let $\bP(\R^{p,q})$ denote the space of lines in $\R^{p,q}$ with action inherited from the action on $\R^{p,q}$ which fixes the first $p-q$ coordinates and acts as multiplication by $-1$ on the last $q$ coordinates. Define a map $\omega\colon\bP(\R^{p,q})\to SO(p,q)$ as follows: 
\[\omega(v) = r(v)r(e_1)\]
where $r(v)$ denotes reflection across the hyperplane orthogonal to $v$ and $e_1=(1,0,\dots,0)\in\R^{p,q}$. Notice that $\omega(v)$ is the product of two reflections, whence a rotation. In addition, the map $\omega$ is a $\Z/2$-equivariant map.

Now, choose a flag $0=V_0\subset V_1 \subset \cdots \subset V_p=\R^{p,q}$ such that $g\cdot V_i=V_i$ and $\dim(V_i\cap V_{i+1})=1$ for all $i=0,1,\dots,p-1$. For example, declaring $V_i=\{(x_1,x_2,\dots,x_i,0,0,\dots,0)\}$ is such a flag. This flag gives rise to a sequence of inclusions $\bP(V_1)\subset \bP(V_2)\subset \cdots \subset \bP(V_p)$. In addition, we can define equivariant maps $\omega\colon \bP(V_i)\to SO_q(\R^p)$ by restricting the map $\omega$ above. For the sake of brevity, write $P^i$ for $\bP(V_i)$ and $P^I$ for $P^{i_1}\times \cdots \times P^{i_m}$ where $I$ is a sequence $(i_1,\dots,i_m)$ with each $i_j<p$. Then we have an equivariant map $\omega\colon P^I\to SO(p,q)$ given by $\omega(v_1,\dots,v_m)=\omega(v_1) \cdots \omega(v_m)$. (The action of $\Z/2$ on $P^I$ is diagonal.) Sequences $I=(i_1, \dots, i_m)$ for which $p>i_1>\cdots>i_m>0$ and the sequence consisting of a single $0$ will be called \bf admissible\rm.

If $\varphi^i\colon D^i \to P^i$ is the characteristic map for the $i$-cell of $P^i$, then the product $\varphi^I\colon D^I\to P^I$ of the appropriate $\varphi^{i_j}$'s is a characteristic map for the top-dimensional cell of $P^I$.

\begin{prop}\label{prop:so} The maps $\omega \varphi^I \colon D^I \to SO(p,q)$, for $I$ ranging over all admissible sequences, are the characteristic maps of a $\mathrm{Rep}(\Z/2)$-complex structure on $SO(p,q)$ for which the map $\omega\colon P^{n-1}\times \cdots \times P^1\to SO(p,q)$ is cellular.
\end{prop}
\begin{proof} The proof is a matter of adapting the proof of the non-equivariant statement in \cite[Proposition 3D.1]{Hatcher}. Consider $SO(p-1,q-1)$ the subset of $SO(p,q)$ which fixes $e_p$. Then Hatcher's maps $p\colon SO(p,q)\to S(\R^{p,p-q})$, $h\colon (P^{p-1}\times SO(p-1,q-1),P^{p-2}\times SO(p-1,q-1))\to (SO(p,q),SO(p-1,q-1))$, and $h^{-1}\colon SO(p,q)-SO(p-1,q-1)\to (P^{p-1}-P^{p-2})\times SO(p-1,q-1)$ are equivariant with the above defined action of $\Z/2$ on $SO(p,q)$.

For the induction, notice that $SO(p-q,q-q)$ has the trivial $\Z/2$-action and so has a $\mathrm{Rep}(\Z/2)$-complex structure using cells which arise from trivial representations. Thus the inductive process begins with $SO(p-q,p-q)$ and continues as in \cite{Hatcher}, and so $SO(p,q)$ has a $\mathrm{Rep}(\Z/2)$-complex structure.

The statement about the map $\omega$ being cellular is also immediate.

\end{proof}

The freeness theorem from \cite{FREENESS} yields the following corollary. (Recall that the coefficient Mackey functor is $\underline{\Z/2}$.)

\begin{cor} $H^{*,*}(SO(p,q))$ is a free $H\underline{\Z/2}$-module.
\end{cor}

\begin{remark}Varying the flag $0=V_0\subset V_1 \subset \cdots \subset V_p=\R^{p,q}$ will alter the cell structure of the projective spaces involved, and hence the cell structure on $SO(p,q)$.
\end{remark}

In light of this remark, it will be convenient to impose a standard cell structure on the real projective spaces. For further convenience, we will restrict attention to the case where $p=n$ and $q=\left\lfloor \frac{n}{2} \right\rfloor$ and, following the notation in \cite{FREENESS}, let $\R\bP^{n}_{tw}=\bP(\R^{n+1,\left\lfloor \frac{n+1}{2}\right\rfloor})$, the equivariant space of lines in $\R^{n+1,\left\lfloor \frac{n+1}{2} \right\rfloor}$.  For example, $\R\bP^3_{tw} = \bP(\R^{4,2})$, $\R\bP^4_{tw} = \bP(\R^{5,2})$, and
$\R\bP^1_{tw} = S^{1,1}$. Considering a Schubert cell decomposition of $\R \bP^{n}_{tw}$ yields the following lemma.

\begin{lemma} $\R \bP^{n}_{tw}$ has a $\mathrm{Rep}(\Z/2)$-structure with cells in dimension $(0,0)$, $(1,1)$, $(2,1)$, $(3,2)$, $(4,2)$, $\dots, (n,\left\lceil \frac{n}{2} \right\rceil)$.
\end{lemma}
\begin{proof}See \cite{FREENESS}.
\end{proof}

With this cell structure, there is an additive basis for $H^{*,*}(\R \bP^{n}_{tw})$ where the bidegrees of the generators agree with the dimensions of the cells.

\begin{prop}
As a $H\underline{\Z/2}$-module, $H^{*,*}(\R \bP^{n}_{tw})$ is free with a
single generator in each degree $(k, \left\lceil \frac{k}{2}
\right\rceil)$ for $k= 0, 1, \dots, n$.
\end{prop}
\begin{proof} See \cite{FREENESS}.
\end{proof}

In this particular case, Proposition \ref{prop:so} indicates that $SO(n,\left\lfloor \frac{n}{2} \right\rfloor)$ has a cell structure with one cell for each admissible sequence $I=(i_1,\dots,i_m)$. These cells are the top cells of the spaces $P^I=\R\bP_{tw}^{i_1}\times \cdots \times \R\bP_{tw}^{i_m}$. Thus, $SO(n,\left\lfloor \frac{n}{2} \right\rfloor)$ has cells in bijection with the cells of $S^{1,1}\times S^{2,1}\times S^{3,2}\times S^{4,2}\times \cdots \times S^{n-1, \left\lceil \frac{n-1}{2} \right\rceil}$.

\begin{example} Consider $SO(2,1)$. The admissible sequences are $(0)$ and $(1)$ which correspond to a $(0,0)$-cell and a $(1,1)$-cell, respectively. Thus $SO(2,1)\cong S^{1,1}$.
\end{example}

\begin{example} Consider $SO(3,1)$. The admissible sequences are $(0)$, $(1)$, $(2)$, and $(2,1)$, which correspond to a $(0,0)$-cell, a $(1,1)$-cell, a $(2,1)$-cell, and a $(3,2)$-cell, respectively. This is consistent with the fact that $SO(3,1)\cong \R\bP_{tw}^3$.
\end{example}

\begin{example} Consider $SO(4,2)$. The admissible sequences give rise to cells with dimensions $(0,0)$, $(1,1)$, $(2,1)$, $(3,2)$, $(3,2)$, $(4,3)$, $(5,3)$, and $(6,4)$. This is consistent with the fact that $SO(4,2)\cong S^{3,2}\times SO(3,1)$.
\end{example}

\begin{example} Consider $SO(5,2)$. The admissible sequences provide a cell structure with cells in dimensions $(0,0)$, $(1,1)$, $(2,1)$, $(3,2)$, $(3,2)$, $(4,2)$, $(4,3)$, $(5,3)$, $(5,3)$, $(6,3)$, $(6,4)$, $(7,4)$, $(7,4)$, $(8,5)$, $(9,5)$, and $(10,6)$. By considering the forgetful long exact sequence, we see that $H^{*,*}(SO(5,2))$ is additively generated by generators with bidegrees matching the dimensions of these cells.
\end{example}

In general, we cannot rely on the forgetful long exact sequence to determine an additive basis for $H^{*,*}(SO(p,q))$. However, the following theorem tells us that there is an additive basis for $H^{*,*}(SO(p,q))$ with generators in bijection with the cells in the above cell structure, and with bidegree agreeing with the dimensions of the cells.

\begin{thm}\label{thm:SOns} $H^{*,*}(SO(n,\left\lfloor \frac{n}{2} \right\rfloor))\cong H^{*,*}(S^{1,1}\times S^{2,1}\times \cdots \times S^{n-1, \left\lceil \frac{n-1}{2} \right\rceil})$, as $H\underline{\Z/2}$-modules.
\end{thm}
\begin{proof} First note that there are no nontrivial differentials in the cellular spectral sequence for the product of projective spaces since there are no nontrivial differentials in the spectral sequences for each individual projective space.

The construction of the cell structure allows for a comparison of the cellular spectral sequence for $SO(n,\left\lfloor \frac{n}{2} \right\rfloor)$ with that of the product of projective spaces. This comparison implies there are no nontrivial differentials in the cellular spectral sequence for $SO(n,\left\lfloor \frac{n}{2} \right\rfloor)$. Explicitly, we can consider the following commutative diagram.
\[ \xymatrix{H^{*,*}(SO(n,\left\lfloor \frac{n}{2} \right\rfloor)^{(k)})  \ar[r]^{\omega^*} \ar[d]^\delta & H^{*,*}((P^I)^{(k)})\ar[d]^{\delta=0} \\
                  H^{*+1,*}(\bigvee_i S^{V_i})  \ar[r]^{\omega^*} & H^{*+1,*}(\bigvee_j S^{V_j})\period}
\]
Here, $X^{(k)}$ denotes the $k$-skeleton, $\delta$ is the connecting homomorphism in the long exact sequence of the pair $(X^{(k+1)},X^{(k)})$, and $\omega^*$ is the map induced by $\omega$. The $V_i$'s and $V_j$'s are $k$-dimensional $\Z/2$-representations. The construction of the cell structure ensures that the lower $\omega^*$ is an isomorphism and the specific choice of cell structure on $P^I$ ensures that the right-hand $\delta$ is zero. Thus, the left-hand $\delta$ must also be zero. These $\delta$'s determine the differentials in the cellular spectral sequence, hence all differentials are zero in the cellular spectral sequence for $SO(n,\left\lfloor \frac{n}{2} \right\rfloor)$.

Since the stated product of spheres has a cell structure with cells of the same dimension as $SO(n,\left\lfloor \frac{n}{2} \right\rfloor)$, the result follows.
\end{proof}

\begin{remark} For general $p$ and $q$, projective spaces other than $\R\bP^n_{tw}$ are involved (e.g. $\bP(\R^{8,2})$, etc.). However, these projective spaces can also be given cell structures for which the cellular spectral sequence has no non-trivial differentials. (See \cite{FREENESS}.) Thus, an analogous statement is true about the spaces $SO(p,q)$, although it is a little more cumbersome to describe the appropriate product of spheres in the general case.
\end{remark}

The isomorphism in the above theorem determines the $H\underline{\Z/2}$-module structure of $H^{*,*}(SO(n,\left\lfloor \frac{n}{2} \right\rfloor))$. We wish to say something about the $H\underline{\Z/2}$-algebra structure. For this, we will need to compare $H^{*,*}(SO(n,\left\lfloor \frac{n}{2} \right\rfloor))$ with $H^{*,*}(P^I)$ where $I=(n-1, n-2, \dots, 1)$ using the map $\omega$ above. We recall the following result from \cite{FREENESS} which describes the $H\underline{\Z/2}$-algebra structure of $\R\bP^\infty_{tw} = \bP(\cU)$ where $\cU \cong (\R^{2,1})^\infty$ is a complete $\Z/2$-universe in the sense of \cite{Alaska}.

\begin{thm}
\label{thm:rpinfty} $H^{*,*}(\R\bP^\infty_{tw})=H\underline{\Z/2}[a,
b]/(a^2=\rho a +\tau b)$, where $\deg(a)=(1,1)$ and $\deg(b)=(2,1)$.
\end{thm}
\begin{proof} See \cite{FREENESS}.
\end{proof}

The natural inclusions $\R\bP^n_{tw} \to \R\bP^\infty_{tw}$ determine the algebra structure for the cohomology of the finite projective spaces.

\begin{thm}\label{thm:rpns} Let $n \geq 2$. Let $deg(a) = (1,1)$ and $deg(b)=(2,1)$.  If $n$ is odd, then $H^{*,*}(\R\bP^n_{tw}) =H\underline{\Z/2}[a, b]/ \sim$ where the generating relations are $a^2 = \rho a + \tau b$ and $b^k =0$ for $k \geq \frac{n+1}{2}$.  If $n$ is even, then $H^{*,*}(\R\bP^n_{tw}) = H\underline{\Z/2}[a ,b]/ \sim$ where the generating relations are $a^2 = \rho a + \tau b$,  $b^k =0$ for $k > \frac{n}{2}$, and $a\cdot b^{n/2}=0$.

If $n=1$, then $H^{*,*}(\R\bP^1_{tw}) = H^{*,*}(S^{1,1}) \cong
H\underline{\Z/2}[a]/(a^2 = \rho a)$.

\end{thm}
\begin{proof} See \cite{FREENESS}.
\end{proof}

The space $P^I$ above is the product of projective spaces $P^I = \R\bP^{n-1}_{tw} \times \cdots \times \R\bP^{1}_{tw}$. In the non-equivariant setting (with constant $\Z/2$ coefficients), $H^*(\R\bP^k)$ is free as a $H^*(pt)$-module ($H^*(pt) \cong \Z/2$), and so the K\"{u}nneth theorem tells us that $H^*( \R\bP^{n-1} \times \cdots \times \R\bP^{1}) \cong  H^*(\R\bP^{n-1}) \otimes \cdots \otimes H^*(\R\bP^{1})$ as $H^*(pt)$-algebras. (Here the tensor products are taken over $H^*(pt)$.) We would like a similar statement about the $RO(\Z/2)$-graded cohomology of the product of projective spaces $P^I$. To do this, we need the following version of the K\"{u}nneth spectral sequence from \cite{UCT_Kunneth}. This K\"{u}nneth theorem is phrased in terms of homological algebra of Mackey functors and views $RO(G)$-graded cohomology as Mackey functor valued, rather than abelian group valued.

\begin{thm}[K\"{u}nneth Theorem \cite{UCT_Kunneth}] Let $\underline{R}_*$ be an $RO(G)$-graded Mackey functor ring which represents the $RO(G)$-graded cohomology theory $\underline{E}^*$. Let $X$ and $Y$ be $G$-spectra indexed on the same universe. There is a natural conditionally convergent cohomology spectral sequence of $\underline{R}^*$-modules
\[
  E^{s, \tau}_2 = \underline{\Ext}^{s, \tau}_{\underline{R}^*} (\underline{R}_{-*}X, \underline{R}^*Y) \Rightarrow \underline{R}^{s + \tau}(X \wedge Y) \period
\]
\end{thm}

%

In addition, if $\underline{R}^*X$ or $\underline{R}^*Y$ is projective, then the spectral sequence collapses at the $E_2$-page and the $E_2$-page can be identified as $\underline{R}^*X \: \Box_{\underline{R}^*} \: \underline{R}^*Y$. The edge homomorphism is then an isomorphism.

The box product $\Box$ for Mackey functors satisfies $(M \: \Box \: N)(G/G) \cong M(G/G) \otimes N(G/G)$. Since each of $H^{*,*}(\R\bP^k_{tw})$ is a free $H\underline{\Z/2}$-module, hence projective, the K\"{u}nneth theorem yields that $H^{*,*}(P^I) \cong H^{*,*}(\R\bP^{n-1}_{tw}) \otimes \cdots \otimes H^{*,*}(\R\bP^{1}_{tw}) $ as $H\underline{\Z/2}$-algebras. (Here the tensor products are taken over $H\underline{\Z/2}$.)

In principal, we can use the map $\omega \colon P^I \to SO(p,q)$ to determine the algebra structure of $H^{*,*}(SO(p,q))$ from the algebra structure of $H^{*,*}(P^I)$. We begin with a few examples.

\begin{example} Consider $SO(4,2)$. We have the map $\omega \colon \R\bP^3_{tw} \times \R\bP^2_{tw} \times \R\bP^1_{tw}$ which is cellular. Because $\omega$ determines the $\mathrm{Rep}(\Z/2)$-structure on $SO(4,2)$ and there are no dimension shifting differentials in the cellular spectral sequence for $SO(4,2)$, the map $\omega^* \colon H^{*,*}(SO(4, 2)) \to H^{*,*}( \R\bP^3_{tw} \times \R\bP^2_{tw} \times \R\bP^1_{tw})$ is injective. In addition, $\omega$ can be thought of as giving an embedding $\R\bP^3_{tw} \to SO(4, 2)$.

Write $H^{*,*}(\R\bP^3_{tw}) = H\underline{\Z/2}[a_3, b_3]/\sim$, $H^{*,*}(\R\bP^2_{tw}) = H\underline{\Z/2}[a_2, b_2]/\sim$, and $H^{*,*}(\R\bP^3_{tw}) = H\underline{\Z/2}[a_1]/\sim$ where each $a_i$ has bidegree $(1, 1)$, each $b_i$ has bidegree $(2, 1)$, and $\sim$ denotes the appropriate equivalence relation as stated in Theorem \ref{thm:rpns}.

  For $i = 1, 2, 3$, let $\beta_i$ be the cohomology generator corresponding to the embedded $i$-dimensional cell of $\R\bP^3_{tw}$. Then $\omega^*(\beta_1) = a_1 + a_2 + a_3$, $\omega^*(\beta_2) = b_2 + b_3$, and $\omega^*(\beta_3) = a_3 b_3$. From Theorem \ref{thm:SOns} we know that $H^{*,*}(SO(4, 2))$ is freely generated with generators in bidegrees $(0, 0)$, $(1, 1)$, $(2, 1)$, $(3, 2)$, $(3, 2)$, $(4, 3)$, $(5, 3)$, and $(6, 4)$. Already we have that $\beta_1$ is a free generator in bidegree $(1, 1)$, $\beta_2$ is a free generator in bidegree $(2, 1)$,  and $\beta_3$ is a free generator in bidegree $(3, 2)$. We wish to see that each of these generators can be expressed as products of the $\beta_i$'s. Computing yields the following:

  \begin{itemize}
  \item $\beta_1^2 = \rho \beta_1 + \tau \beta_2$,
  \item $\beta_1^3 = \rho \beta_1^2 + \tau \beta_1 \beta_2 \neq 0$, from which it can be inferred that $\beta_1 \beta_2$ is a free generator in bidegree $(3, 2)$,
  \item $\beta_2^2 = 0$,
  \item $\beta_1 \beta_3$ is a free generator in bidegree $(4,3)$,
  \item $\beta_2 \beta_3$ is the free generator in bidegree $(5, 3)$,
  \item $\beta_3^2 = 0$,
  \item $\beta_1 \beta_2 \beta_3$ is the free generator in bidegree $(6, 4)$.
  \end{itemize}
  
  These results are summarized in the proposition below.
\end{example}

\begin{prop} As an $H\underline{\Z/2}$-algebra, $H^{*,*}(SO(4, 2)) \cong H\underline{\Z/2}[\beta_1, \beta_2, \beta_3]/\sim$, where $\beta_i$ is in bidegree $(i, \left \lceil i/2 \right \rceil)$ and $\sim$ is generated by $\beta_1^2 = \rho \beta_1 + \tau \beta_2$, $\beta_2^2 = 0$, and $\beta_3^2 = 0$.
\end{prop}

\begin{remark} It is worthwhile to notice that $\psi(H^{*,*}(SO(4, 2))) \cong \Z/2[\beta_1, \beta_3]/(\beta_1^3, \beta_3^2)$ and this is precisely $H^*(SO(4); \Z/2)$. (Notice that $\psi(\beta_2) = \psi(\beta_1^2)$.)
\end{remark}

\begin{example} Consider $SO(5, 2)$. From Theorem \ref{thm:SOns} we know that $H^{*,*}(SO(5, 2))$ is freely generated with generators in bidegrees $(0, 0)$, $(1, 1)$, $(2, 1)$, $(3, 2)$, $(3, 2)$, $(4,2)$, $(4, 3)$, $(5, 3)$, $(5, 3)$, $(6, 3)$, $(6, 4)$, $(7, 4)$, $(7, 4)$, $(8, 5)$, $(9, 5)$, and $(10, 6)$.

Again, we'll write $H^{*,*}(\R\bP^4_{tw}) = H\underline{\Z/2}[a_4, b_4]/\sim$, $H^{*,*}(\R\bP^3_{tw}) = H\underline{\Z/2}[a_3, b_3]/\sim$, $H^{*,*}(\R\bP^2_{tw}) = H\underline{\Z/2}[a_2, b_2]/\sim$, and $H^{*,*}(\R\bP^3_{tw}) = H\underline{\Z/2}[a_1]/\sim$ where each $a_i$ has bidegree $(1, 1)$, each $b_i$ has bidegree $(2, 1)$, and $\sim$ denotes the appropriate equivalence relation as stated in Theorem \ref{thm:rpns}.

For $i = 1, \dots, 4$, let $\beta_i$ be the cohomology generator corresponding to the embedded $i$-dimensional cell of $\R\bP^4_{tw}$. Then $\omega^*(\beta_1) = a_1 + a_2 + a_3 + a_4$, $\omega^*(\beta_2) = b_2 + b_3 + b_4$, $\omega^*(\beta_3) = a_3 b_3 + a_4 b_4$, and $\omega^*(\beta_4) = b_4^2$.

As in the previous example, we would like to see that the remaining free generators can be expressed as products of these $\beta_i$'s. We compute:

\begin{itemize}
\item $\beta_1$ is a generator in bidegree $(1,1)$,
\item $\beta_2$ is a generator in bidegree $(2,1)$,
\item $\beta_1^2 = \rho \beta_1 + \tau \beta_2$,
\item $\beta_3$ is a generator in bidegree $(3,2)$,
\item $\beta_1 \beta_2$ is a free generator in bidegree $(3,2)$,
\item $\beta_2^2 = \beta_4$ is a free generator in bidegree $(4, 2)$,
\item $\beta_1 \beta_3$ is a free generator in bidegree $(4,3)$,
\item $\beta_1 \beta_2^2$ is a free generator in bidegree $(5,3)$,
\item $\beta_2 \beta_3$ is a free generator in bidegree $(5,3)$,
\item $\beta_2^3$ is a free generator in bidegree $(6,3)$,
\item $\beta_3^2 = 0$,
\item $\beta_1 \beta_2 \beta_3$ is a free generator in bidegree $(6,4)$,
\item $\beta_1 \beta_2^3$ is a free generator in bidegree $(7,4)$,
\item $\beta_2^2 \beta_3$ is a free generator in bidegree $(7,4)$,
\item $\beta_2^4 = 0$,
\item $\beta_1 \beta_2^2 \beta_3$ is a free generator in bidegree $(8,5)$,
\item $\beta_2^3 \beta_3$ is a free generator in bidegree $(9,5)$,
\item $\beta_1 \beta_2^3 \beta_3$ is a free generator in bidegree $(10,6)$
\end{itemize}
We condense these results in the proposition below.
\end{example}

\begin{prop} As an $H\underline{\Z/2}$-algebra, $H^{*,*}(SO(5, 2)) \cong H\underline{\Z/2}[\beta_1, \beta_2, \beta_3]/\sim$, where $\beta_i$ is in bidegree $(i, \left \lceil i/2 \right \rceil)$ and $\sim$ is generated by $\beta_1^2 = \rho \beta_1 + \tau \beta_2$, $\beta_2^4 = 0$, and $\beta_3^2 = 0$.
\end{prop}

In general, we have the following result.

\begin{thm}\label{thm:SOpq_algebra}  Let $p>1$ and $q=\lfloor p/2 \rfloor$. Then as an $H\underline{\Z/2}$-algebra, 
\[H^{*,*}(SO(p, q)) \cong \left(H\underline{\Z/2}[\beta_1, \beta_2]/\langle \beta_1^2 = \rho \beta_1 + \tau \beta_2, \; \beta_2^{n_2}\rangle \right) \otimes \displaystyle\bigotimes_{i \geq 3,\; i \mathrm{\;odd}}H\underline{\Z/2}[\beta_i]/ \langle \beta_i^{n_i} \rangle \]
where $\beta_i$ has bidegree $(i, \lceil i/2 \rceil)$ and $n_i$ is the smallest power of $2$ such that $i \cdot n_i \geq p$ for $i \geq 2$.

\end{thm}

\begin{proof} Write $H^{*,*}(\R\bP^k_{tw}) \cong H\underline{\Z/2}[a_k, b_k]/\sim$ as in Theorem \ref{thm:rpns}.
For $i = 1, \dots, p-1$, let $\beta_i$ be the cohomology generator corresponding to the embedded $i$-dimensional cell of $\R\bP^{p-1}_{tw}$. Then $\omega^*(\beta_i) =\displaystyle\sum_{j=i}^{p-1}a_j b_j^{(i-1)/2}$ if $i$ is odd and  $\omega^*(\beta_i) =\displaystyle\sum_{j=i}^{p-1}b_j^{i/2}$ if $i$ is even. In particular, $\omega^*(\beta_1) = \displaystyle\sum_{j=1}^{p-1}a_j$ and so  $\omega^*(\beta_1^2) = \displaystyle\sum_{j=i}^{p-1}(a_j)^2 = \sum_{j=i}^{p-1}\rho a_j + \tau b_j = \omega^*(\rho \beta_1 + \tau \beta_2)$. Since $\omega^*$ is an injection, $\beta_1^2 = \rho \beta_1 + \tau \beta_2$. Similarly, if $i>1$ then $\beta_i^2 = \beta_{2i}$ if $2i < p$ and $\beta_i^2 = 0$ if $2i \geq p$.

Let $A$ be the $H\underline{\Z/2}$-algebra $H\underline{\Z/2}[\beta_1, \beta_2, \dots]/\sim$ where the relations are  $\beta_1^2 = \rho \beta_1 + \tau \beta_2$, $\beta_i^2 = \beta_{2i}$ if $2i < p$, and $\beta_i^2 = 0$ if $2i \geq p$. The preceding observations are enough to see that there is a surjective map $A \to H^{*,*}(SO(p, q))$.  The relations in $A$ allow each element of $A$ to be expressed as a linear combination of monomials $\beta_I = \beta_{i_1} \beta_{i_2} \cdots \beta_{i_n}$ for admissible sequences $I$. These monomials are linearly independent in $H^{*,*}(SO(p, q))$, hence also in $A$, and so the map $A \to H^{*,*}(SO(p, q))$ is an isomorphism.

The relations $\beta_i^2 = \beta_{2i}$ and $\beta_1^2 = \rho \beta_1 + \tau \beta_2$ allow the admissible monomials $\beta_I$ of $A$ to be uniquely expressed in terms of $\beta_1$, $\beta_2$, and $\beta_i$ for $i \geq 3$ and odd. The relation $\beta_j = 0$ for $j \geq p$ can be written as $\beta_{i n_i} = \beta_i ^{n_i} = 0$ where $j = i n_i$ with $i$ odd and $n_i$ a power of 2. This relation holds if and only if $i \cdot n_i \geq p$. Hence, $A$ can be expressed as the tensor product in the statement of the theorem.
\end{proof}

\section{Stiefel manifolds}\label{sec:stiefel}
Let $V_k(\R^{p,q})$ denote the Stiefel manifold of $k$-frames in $\R^{p,q}$ with action inherited from the one on $\R^{p,q}$. There is an equivariant projection $\pi \colon O(p,q)\to V_{q}(\R^{p,q})$ sending $A\in O(p,q)$ to the $q$-frame consisting of the last $q$ columns of $A$.

For simplicity, we will restrict to the case where $p>1$ and $q=\left\lfloor \tfrac{p}{2} \right\rfloor$. The projection $\pi \colon SO(p,q)\to V_{q}(\R^{p,q})$ is surjective and we can view $V_{q}(\R^{p,q})$ as the coset space $SO(p)/SO(p-q)$ with action inherited from $SO(p,q)$. From this viewpoint, we can give $V_{q}(\R^{p,q})$ a $\mathrm{Rep}(\Z/2)$-complex structure. The cells are the sets of cosets corresponding to admissible sequences $I=(i_1,\dots,i_m)$ where $p > i_1 > \cdots > i_m \geq p-q$. Since the Stiefel manifold has a $\mathrm{Rep}(\Z/2)$-complex structure, its cohomology is free as a $H\underline{\Z/2}$-module. The additive structure of the cohomology of the Stiefel manifold is captured by the following theorem.

\begin{thm} Let $p>1$ and $q=\left\lfloor \tfrac{p}{2} \right\rfloor$. Then $H^{*,*}(V_{q}(\R^{p,q})) \cong H^{*,*}(S^{p-q, \left\lceil \frac{p-q}{2} \right\rceil}\times \cdots \times S^{p-1, \left\lceil \frac{p-1}{2} \right\rceil})$, as $H\underline{\Z/2}$-modules.
\end{thm}
\begin{proof} The additive cohomology generators of $H^{*,*}(S^{p-q,\left\lceil \frac{p-q}{2} \right\rceil}\times \cdots \times S^{p-1, \left\lceil \frac{p-1}{2} \right\rceil})$ are in bijection with the admissible sequences $I=(i_1,\dots,i_m)$ where $p > i_1 > \cdots > i_m \geq p-q$. Comparison with $SO(p,q)$ shows that there are no nontrivial differentials in the cellular spectral sequence for $V_{q}(\R^{p,q})$. Thus, $H^{*,*}(V_{q}(\R^{p,q}))$ is free with generators in bijection with the cells and with bidegrees agreeing with the dimensions of these cells.
\end{proof}

Following \cite{Miller}, we will denote by $[i_1, \dots, i_n]$ the cohomology generator corresponding to the admissible sequence $I=(i_1, \dots, i_n)$. The previous theorem implies that these classes form an additive basis for $H^{*,*}(V_{q}(\R^{p,q}))$. We further make the conventions that $[i_1, \dots, i_n] = [i_{\lambda(1)}, \dots, i_{\lambda(n)}]$ for any permutation $\lambda$ of the indices, and that $[i_1, \dots, i_n] = 0$ if some $i_k < p-q$, if some $i_k \geq p$, or if some $i_k = i_j$ for $k \neq j$. We also denote by $[0]$ the generator corresponding to the admissible sequence $I=(0)$.\footnote{Actually, \cite{Miller} uses the notation $(i_1, \dots, i_n)$ rather than $[i_1, \dots, i_n]$. We use the square brackets here to avoid confusion, since the parentheses are already overused.}

\begin{thm}\label{thm:stiefel_algebra} Let $p > 2$ and $q = \lfloor p/2 \rfloor $. As an $H\underline{\Z/2}$-algebra, $H^{*,*}(V_{q}(\R^{p,q}))$ is multiplicatively generated by $[0]$ and all $[i]$ with $p > i \geq p-q$ subject only to the relations
\begin{itemize}
\item $[0]$ is the unit, and
\item $[i] \cup [j] = \begin{cases} [i, j] & \text{if } i + j < p, \\ 0 &  \text{if } i + j \geq p\period \end{cases}$
\end{itemize} 
\end{thm}

\begin{proof} The map $\pi \colon SO(p,q)\to V_{q}(\R^{p,q})$ is by definition cellular and we can compare the cellular spectral sequences for each space to see that $\pi$ induces an injection $\pi^* \colon H^{*,*}(V_{q}(\R^{p,q})) \to H^{*,*}(SO(p,q))$.
For each admissible sequence $I= (i_1, \dots, i_n)$, $\pi^*([i_1, \dots, i_n]) = \beta_{i_1} \cdots \beta_{i_n}$. (Since we are assuming $p > 2$, none of the $i_k$'s are 1's and we need not concern ourselves with the relation $\beta_1^2 = \rho \beta_1 + \tau \beta_2$.) In particular, $\pi^*([i] \cup [j]) = \beta_i \beta_j = \pi^*([i, j])$, and so $[i] \cup [j] = [i, j]$. This class is zero if and only if $i + j \geq p$. The assumptions on $p$ and $q$ force $2i \geq p$ since $i \geq p - \lfloor p/2 \rfloor$, so $[i] \cup [i] = 0$ for all $i$.
\end{proof} 


Notice that \cite{Miller} provides more detail about the product structure in the non-equivariant setting. Namely, $[i] \cup [j_1, \dots, j_n] = [i, j_1, \dots, j_n] + \sum_k [j_1, \dots, j_k + i, \dots, j_n]$ in $H^*(V_q(\R^p))$. However, in the setting of the theorem, $i$ and all of the $j_k's$ are between $p$ and $p-\lfloor p/2 \rfloor$. Thus $i + j_k \geq p$ and so the summation term above is zero. In addition we necessarily have that $[i] \cup [i] = [2i] = 0$ in the non-equivariant setting, as mentioned in the proof above.

If $p = 2$, then the Stiefel manifold $V_1(\R^{2,1})$ can be identified with $S(\R^{2,1}) \cong S^{1,1}$ and the cohomology of this space has already been determined.

\section{Further directions}\label{sec:further}
In \cite{Miller}, Miller computes the Steenrod algebra structures of $H^*(V_k(\R^n))$ and $H^*(SO(n))$. The algebra of cohomology operations in ordinary $RO(\Z/2)$-graded cohomology with constant  $\underline{\Z/2}$-coefficients is not yet fully understood. It is shown in \cite{Caruso} that the image of the integer-degree stable equivariant operations in the non-equivariant Steenrod algebra consists only of the identity and the Bockstein. However, this author is not aware of any results about the other, non-integer-degree operations. In light of the similarities between $H^{*,*}(V_q(\R^{p,q}))$ and $H^*(V_q(\R^p))$ when $p > 1$ and $q = \lfloor p/2 \rfloor$, one might expect that the equivariant Steenrod algebra acts analogously to the non-equivariant Steenrod algebra, at least in the case of the Stiefel manifolds studied in this paper. Currently, not enough is known about the equivariant Steenrod algebra, and so there is no way of making the previous statement precise and investigating the truth of it.

Both Stiefel manifolds and rotation groups are related to the oriented and unoriented Grassmann manifolds of $k$-planes in $\R^n$ in an essential way. Non-equivariantly, one has fibre bundles $O(k) \to V_k(\R^n) \to G_k(\R^n)$. In addition, the unoriented Grassmann manifolds can be viewed as homogeneous spaces via the identification $O(n)/O(k) \times O(n-k) \cong G_k(\R^n)$. Similarly, the oriented Grassmann manifolds are $SO(n)/SO(k) \times SO(n-k) \cong \widetilde{G}_k(\R^n)$. All of these constructions carry over to the equivariant setting. Thus, the computations above should lead to information about the cohomology of $G_q(\R^{p,q})$, the unoriented Grassmannian of $q$-planes in $\R^{p, q}$ with action inherited from $\R^{p, q}$, and also the cohomology of $\widetilde{G}_q(\R^{p,q})$. This would be an enormous boon to the equivariant theory as the equivariant Grassmannians classify equivariant vector bundles. Difficulties arise since these quotients are not cellular and so the above techniques cannot be applied. In addition, the Serre spectral sequence for a $\Z/2$-fibration (see \cite{ROGSS}) requires the use of local coefficients when the base space is not equivariantly 1-connected, as is the case with the Grassmannians. An attempt at adapting the techniques of \cite{Shulman} to this setting would be quite interesting.

In addition, it would be interesting to see which of the results in \cite{James} related to Stiefel manifolds can be extended in a meaningful way to the $\Z/2$-equivariant setting.



\bibliography{references}
\bibliographystyle{amsalpha}

\end{document}